\documentclass[12pt]{amsart}

\usepackage{amssymb,url,mathrsfs,euscript}

\newtheorem{teo}{Theorem}[section]
\newtheorem{lema}{Lemma}[section]
\newtheorem{pro}{Proposition}[section]
\newtheorem{defi}{Definition}[section]

\newtheorem{rmk}{Remark}[section]\newtheorem*{rmk*}{Remark}

\newtheorem{coro}{Corollary}[section]

\theoremstyle{remark}     
\newtheorem*{ex*}{Example}
\newtheorem*{exs*}{Examples}

\newtheorem*{acknowledgements}{Acknowledgements}

\def\sideremark#1{\ifvmode\leavevmode\fi\vadjust{\vbox to0pt{\vss
\hbox to 0pt{\hskip\hsize\hskip1em%
\vbox{\hsize2cm\tiny\raggedright\pretolerance10000%
\noindent #1\hfill}\hss}\vbox to8pt{\vfil}\vss}}}%

\swapnumbers              
\theoremstyle{plain}      
\newtheorem{lemma}[rmk]{Lemma}

\newtheorem{theorem}[rmk]{Theorem}

\theoremstyle{definition} 

\newcommand{\bt}{\begin{theorem}}\newcommand{\et}{\end{theorem}}
\newcommand{\bl}{\begin{lemma}}\newcommand{\el}{\end{lemma}}
\newcommand{\bp}{\begin{proof}}\newcommand{\ep}{\end{proof}}

\newcommand{\be}{\begin{equation}}\newcommand{\ee}{\end{equation}}
\newcommand{\bdm}{\begin{displaymath}}
\newcommand{\edm}{\end{displaymath}}

\numberwithin{equation}{section}

\def \z{\zeta}

\newcommand{\D}{\curly{D}}
\renewcommand{\geq}{\geqslant}\renewcommand{\leq}{\leqslant}
\newcommand{\R}{\mathbb{R}}

\DeclareMathOperator{\im}{Im}

\renewcommand{\span}{\textsl{span}}
\newcommand{\grad}{\textsl{grad\,}}

\newcommand{\curly}{\mathscr}

\newcommand{\bb}{\mathbb}
\newcommand{\ba}{\begin{array}}\newcommand{\ea}{\end{array}}

\renewcommand{\&}{{\footnotesize \&}}

\begin{document}

\title[]{A splitting theorem for higher order parallel immersions}
\subjclass[2000]{}
\date{\today}
\author[I. Kath]{Ines Kath}
\author[P.-A. Nagy]{Paul-Andi Nagy} 
\address{Institut f\"ur Mathematik und Informatik, Ernst-Moritz-Arndt Universit\"at Greifs\-wald, Walther-Rathenau str.47, 
17487 Greifswald, Germany}
\email{ines.kath@uni-greifswald.de, nagy@math.auckland.ac.nz}

\frenchspacing

\begin{abstract} We consider isometric immersions into space forms having the second fundamental form parallel at order $k$. We show that this class of immersions consists 
of local products, in a suitably defined sense, of parallel immersions and normally flat immersions of flat spaces. 
\end{abstract}

\date{\today}
\maketitle

\section{Introduction} \label{intro}
A basic question in submanifold geometry is the study of isometric immersions of Riemannian manifolds having their second fundamental form subject to certain geometric or analytic properties. One of the best known examples is the class of immersions having parallel second fundamental form. Such immersions are called parallel. Their study 
is part of well established theories, for instance the Euclidean case has been fully described by Ferus \cite{F1, F2, F3}. 
Parallel immersions in space forms are the extrinsic counterpart of locally symmetric spaces \cite{Str}. 

In this spirit, a natural class to consider is that of {\it{$k$-parallel}} immersions, by requiring the second fundamental form be parallel at order $k$ for $k \geq 1$. At a pure analogy level, while each Riemannian manifold whose Riemannian curvature tensor is parallel at higher order is locally symmetric \cite{K-N} the class of $k$-parallel immersions is sensibly larger than that of parallel ones. Curves such as the Cornu spiral provide simple $2$-parallel, non-parallel examples.

Early work by Mirzoyan \cite{Mir78} on $k$-parallel immersions has been taken up by Dillen and Lumiste among others, see \cite{Lumbook} for an overview. In the case of immersions into space forms classification results have been obtained for low dimension or codimension or for small $k$ in 
\cite{L1,D1, Lum86, Lum00,Lum89, Lum90, Lum91}.  A case by case 
inspection therein reveals that the immersion is a product and each factor is either a $k$-parallel curve, a $k$-parallel flat and normally flat surface, an affine subspace or a sphere. 

That this reflects the general structure of $k$-parallel immersions into space forms is the main result of this note. We prove\\[1ex] 
{\bf Theorem } {\it
Any  $k$-parallel isometric immersion of a Riemannian manifold $(M,g)$ into a simply-connected space form is the local product 
of a parallel immersion and a $k$-parallel immersion of a flat space with flat normal bundle.}\\[1ex]

This is Theorem \ref{main1} in the paper. Products of immersions are understood in the sense of Noelker \cite{Noelker}, see Definition~\ref{def-1}. Open manifolds are naturally considered since in the compact case integration shows that $k$-parallel immersions are parallel.

To prove the theorem we first note that having the ambient space locally symmetric forces the Riemann curvature  
be parallel at higher order hence parallel by \cite{K-N}.
We refine the natural local product decomposition of $(M,g)$ 
into a flat space and a locally symmetric space with non-degenerate curvature. The key ingredient is to take into account 
the algebraic structure of the normal bundle reflecting the structure of higher order derivatives of the second fundamental form. It is used to modify the above mentioned splitting into one that satisfies the requirements 
in the generalised Moore splitting criterion \cite{Mol,Noelker}.

Our theorem reduces the study of $k$-parallel immersions into space forms 
to that of parallel immersions and $k$-parallel normally flat immersions of flat space. 
This is an essential step towards the classification of $k$-parallel immersions of space forms. Indeed, parallel immersions of space forms are known. The Euclidean case was solved by Ferus as already mentioned above. If the target space is a sphere one can easily get a classification from the Euclidean case. For the hyperboloid as target a classification was achieved independently by Takeuchi \cite{Tak} and Backes and Reckziegel \cite{BR}.
Furthermore, in \cite{Lum91} Lumiste describes a general method (the so-called polynomial map method) for the investigation of $k$-parallel immersions in Euclidean spaces that are flat and normally flat; see also  \cite{Mokhov} where the latter are interpreted as certain integrable systems of hydrodynamic type. In conclusion, to obtain a full classification there remains to understand the structure of flat, normally flat $k$-parallel immersions in hyperbolic space. 

\section{Structure results}
Let $(N^n,g)$ be a Riemannian manifold and let us consider an isometric immersion $(M^m,g) \hookrightarrow (N,g), m \leq n $. We will denote by $\nabla$ and $\nabla^N$, respectively, the Levi-Civita connections and 
by $R, R^N$ the Riemann curvature tensors, with the convention that $R^N(X,Y,Z,U)=g((\nabla^N)^{2}_{Y,X}Z-(\nabla^N)^{2}_{X,Y}Z,U)$ for all $X,Y,Z,U$ in $TN$. The orthogonal splitting 
$$ TN=TM \oplus NM
$$
along $M$, enables to consider the second fundamental form $\alpha :TM \times TM \to NM$, where $NM$ is the normal bundle of $M$. We denote by $\nabla^\perp$ the induced covariant derivative in $NM$ and by $R^\perp$ the corresponding curvature tensor. Furthermore, we write $\nabla \alpha$ for the covariant derivative of $\alpha$, which is defined by $\nabla$ and $\nabla^\perp$. 
\begin{defi} \label{k-p}
 Let $(M,g) \hookrightarrow (N,h)$ be an isometric immersion. It is called $k$-parallel for some $k \geq 1$ if and only if 
\begin{equation} \label{kpe}
\nabla^k \alpha=0.
\end{equation}
\end{defi}
A particular case thereof consists in the so-called {\it{parallel}} immersions, when one requires $\nabla \alpha=0$. 

In most of what follows we will look at immersions  into standard spaces $M^n(c)$; that is simply connected, complete spaces of constant sectional curvature $c$ and dimension $n$, realised by their standard models in $\bb{R}^{n+1}$: the sphere, the hyperboloid and the hyperplane.

The main goal of this paper is to prove that the study of $k$-parallel immersions into standard spaces reduces 
to that of parallel immersions and $k$-parallel immersions of flat space. The latter class will be shown later on to be normally flat as well. 

One of our main tools is Moore's decomposition criterion \cite{Moore} and its generalisations to the case of standard spaces due to \cite{Mol} (see also \cite{Noelker}). Its application requires setting up the notion of product immersion which we recall below in the Euclidean case, to begin with.
\begin{defi} \label{def-0}
 An isometric immersion $f:M_1 \times M_2 \to \bb{R}^n$ of a Riemannian product is called a product immersion if and only if 
$f=F \circ (f_1 \times f_2)$ for isometric immersions $f_i:M_i \to N_i, i=1,2$ where $N_1,N_2$ are affine sub-spaces of 
$\bb{R}^n$ such that $F:N_1 \times N_2 \to \bb{R}^n$ is isometric.
\end{defi}

The next definition captures necessary conditions for an immersion to be a product.
In the rest of this paper we will systematically use the natural identification $T(M_1 \times M_2)=TM_1 \oplus TM_2$ whenever $M_1,M_2$ are smooth manifolds.
 
\begin{defi}The second fundamental form $\alpha$ of an isometric immersion $f:M_1 \times M_2 \to M$ is called 
decomposable if $\alpha(\D_1,\D_2)=0$ where $\D_{i}=df(TM_i), i=1,2$.
\end{defi}
Having decomposable second fundamental has been shown, for Euclidean target spaces, to yield a product structure in \cite{Moore}; explicitly
\begin{teo} \label{M1}
An isometric immersion $f:M_1 \times M_2 \to \bb{R}^n$ where $M_1,M_2$ are connected and $f$ has decomposable second fundamental form is a product immersion.
\end{teo}
To present the generalisation of the above result to standard space target we first recall that a submanifold 
$N$ in $M^n(c)$ is called spherical if its second fundamental form $\alpha=g \otimes \z$ for some parallel normal vector field, where $g$ is the induced metric on $N$. For $c \neq 0$ this is equivalent with $N$ being the intersection of the standard space and some affine subspace of $\bb{R}^{n+1}$.  
\begin{defi} \label{def-1}
 An isometric immersion $f:M_1 \times M_2 \to M^n(c)$ of a Riemannian product is called a product immersion if and only if 
$f=F \circ (f_1 \times f_2)$ for isometric immersions $f_i:M_i \to N_i, i=1,2$ where:
\begin{itemize}
\item[(i)] $N_1,N_2$ are isometric to standard spaces and admit isometric embeddings $\varphi_i:N_i \to M^n(c), i=1,2$ with spherical 
image such that  $\varphi_1(N_1) \cap \varphi_2(N_2)=\{\overline{p}\}$;
\item[(ii)] the map $F:N_1 \times N_2 \to \bb{R}^{n+1},\ (p_1,p_2) \mapsto \overline{p}+
(\varphi_1(p_1)-\overline{p})+(\varphi_2(p_2)-\overline{p})$ is an
isometric embedding with image contained in $M^n(c)$. 
\end{itemize}
\end{defi}

The curved counterpart of Moore's decomposition criterion in Theorem \ref{M1} is
\begin{teo} \label{M2}\cite{Mol} {\rm (see also \cite{Noelker})}
 An isometric immersion $f:M_1 \times M_2 \to M^n(c)$ where $M_1,M_2$ are connected and $f$ has decomposable second fundamental form is a product immersion.
\end{teo}
Whereas Definition \ref{def-0} is a particular case of Definition \ref{def-1}, the latter allows a broader treatment of the case when $c=0$ e.g. when $N_1, N_2$ are taken to be spheres.

An isometric immersion $f:M \to N$ will be called a local product if for each $x$ in $M$ there exists a product open 
set $M_1 \times M_2$ around $x$ such that $f_{\vert M_1 \times M_2}$ is a product immersion.

The main result in this paper can now be formulated as follows. 
\begin{teo} \label{main1}Let $(M,g) \rightarrow M^n(c)$ be a $k$-parallel isometric immersion. It is the local product 
of a parallel immersion 
and a $k$-parallel immersion of a flat space 
with flat normal bundle. When $(M,g)$ is simply connected and complete the splitting is global.
\end{teo}

The proof requires a few technical steps we will outline below.  When not specified otherwise we will work under the assumptions of Theorem \ref{main1}.

We define at each point of $M$
$$  E_0=\{X \in TM: R(TM,TM)X=0 \}, \  E_1=E_0^{\perp} $$
and note that $E_1=\span \{ R(X,Y)Z : X,Y,Z \in TM\} $.
\begin{pro}\label{s1}
Let $(M,g) \hookrightarrow (N,h)$ be a $k$-parallel isometric immersion, where $(N,h)$ is locally symmetric. The following hold:
\begin{itemize}
\item[(i)] $E_0$ and $E_1$ are parallel in $TM$; 
\item[(ii)] the distribution $E_0$ is flat; 
\item[(iii)] $E_1 \subset \{X \in TM: \nabla_X \alpha=0\}.$
\end{itemize}
\end{pro}
\begin{proof}
(i) We differentiate the Gauss equation 
\begin{equation} \label{g1}
R^N(X_1,X_2,X_3,X_4)=R(X_1,X_2,X_3,X_4)+\langle \alpha_{X_2}X_3, \alpha_{X_1}X_4\rangle-\langle \alpha_{X_1}X_3, \alpha_{X_2}X_4\rangle
\end{equation}
where $X_i, 1 \le i \le 4$, belong to $TM$ to find that $$(\nabla^{2k-1}R)(X_1,X_2,X_3,X_4)=(\nabla^{2k-1}R^N)(X_1,X_2,X_3,X_4)$$ in directions tangent to $M$. Now, taking into account that $\nabla^NR^N=0$ yields 
$$ -(\nabla_{X_0}R^N)(X_1,X_2,X_3,X_4)=\sum \limits_{i=1}^4 R^N(X_1, \ldots, \alpha_{X_0}X_i, \ldots, X_4)
$$
for all $X_0$ in $TM$. Now the Codazzi-Mainardi formula 
\begin{equation}\label{c-m}
-R^N(X_1,X_2,X_3,\xi)=\langle (\nabla_{X_1}\alpha)_{X_2}X_3-(\nabla_{X_2}\alpha)_{X_1}X_3,\xi \rangle 
\end{equation}
for $\xi\in NM$ makes that, for instance, 
$$ -R^N(X_1,X_2,X_3,\alpha_{X_0}X_4)=\langle (\nabla _{X_1}\alpha)_{X_2}X_3- (\nabla _{X_2}\alpha)_{X_1}X_3, \alpha_{X_0}X_4\rangle
$$
hence clearly $\nabla^{2k-1}R^N=0$ and then $\nabla^{2k-1}R=0$. By \cite{K-N} we have that $\nabla R=0$ 
whence the claim.\\
(ii) follows from the definition of $E_0$.\\
(iii) We exploit an argument used in \cite{K-N, Tanno} under slightly different assumptions. The function $t=\vert \alpha \vert^2:M\rightarrow \R$ satisfies  
\begin{equation} \label{g3}
\langle \nabla_{X}\grad t, Y \rangle=2\langle \nabla_X \alpha, \nabla_Y \alpha \rangle
\end{equation}
for all $X,Y$ in $TM$; it follows that $\nabla^{2k-2} \grad t=0$ along $M$, after differentiating at higher order. In particular $\nabla^{2k-2}X_1=0$
where $X_1$ denotes the orthogonal projection of $\grad t$ onto the parallel distribution $E_1$. Now we can use the following  result due to Tanno.  Suppose that at some point $x$ of a Riemannian manifold and for some tangent vectors $X,Y$ at $x$, $R(X,Y)$ is not singular. Then having $\nabla^k T=0$ for an arbitrary tensor $T$ and for some $k\ge 1$ implies $\nabla T=0$,  see \cite{Tanno}, Theorem 1. Applied to the integral manifolds of $E_1$ this
yields $\nabla_UX_1=0$ for all $U$ in $E_1$; in particular $R(E_1,E_1)X_1=0$. Since, moreover, $R(E_0,TM)=0$
we get $R(TM,TM)X_1=0$, thus $X_1\in E_0$ showing that $X_1=0$. In other words, $\grad t$ belongs to $E_0$, 
which is parallel in $TM$. When taking $X,Y$ in $E_1$ it follows
that the right hand side of \eqref{g3} vanishes and the claim follows by a positivity argument.
\end{proof}
Under additional assumptions we have the following alternative, by using essentially that 
a Ricci flat, locally symmetric space is flat.
\begin{coro} \label{e}

If we have a $k$-parallel isometric immersion of an Einstein or locally irreducible manifold $M$ into a locally symmetric space, then either the immersion is parallel or $M$ is flat.
\end{coro}
In what follows the action of the curvature tensor on a tensor field $Q$ in $(T^{\star}M)^l \otimes NM$ is given by 
\begin{equation*}
 (R(X,Y)\cdot Q)(Z_1, \ldots, Z_l)=R^{\perp}(X,Y)Q(Z_1, \ldots, Z_l)-\sum \limits_{i=1}^{l}Q(Z_1, \ldots, R(X,Y)Z_i, \ldots, X_l)
\end{equation*}
whenever $X,Y, Z_i, 1 \le i \le l$ belong to $TM$. For $s \geq 0$ we will use frequently the short 
hand notation 
$\im \nabla^s \alpha$ to refer to $\span \{\nabla^s_{X_1, \ldots, X_s}\alpha \} \subseteq NM$
where $X_i, 1 \le i \leq s$ belong to $TM$.

\begin{lema} \label{ll1} Let $(M,g)$ be a $k$-parallel isometric immersion into a space form. We have 
\begin{itemize}
\item[(i)] $R^{\perp}(E_0,TM)(\im \nabla^{k-s} \alpha)=0, s \geq 0;$
\item[(ii)] $R^{\perp}(E_0,TM)=0.$
\end{itemize}
\end{lema}
\begin{proof}
Because $\nabla^2(\nabla^{k-2} \alpha)=0$ it follows after anti-symmetrisation that the curvature operator $R(X,Y)$ acts trivially on $\nabla^{k-2} \alpha$; since moreover $R(E_0,TM)TM=0$ we obtain that 
$$ R^{\perp}(E_0,TM) (\im \nabla^{k-2} \alpha)=0.
$$
Further anti-symmetrisation in the first two arguments in $\nabla^{k-2} \alpha=\nabla^2(\nabla^{k-4} \alpha)$ yields 
$$ R^{\perp}(E_0,TM) (R^{\perp}(E_0,TM) (\im \nabla^{k-4} \alpha))=0.
$$
After taking scalar products with vectors in $\im \nabla^{k-4} \alpha$ a positivity argument shows that 
$$R^{\perp}(E_0,TM) (\im \nabla^{k-4} \alpha)=0.$$ Continuing this procedure leads to  
$R^{\perp}(E_0,TM)  (\im \nabla^{k-2s} \alpha)=0, s \geq 0$. Since $k$-parallel manifolds are also $k+1$-parallel this equation holds also for $k+1$ and (i) follows.

In particular, (i) gives $R^{\perp}(E_0,TM)  (\im \alpha)=0$. We will now use the Ricci formula 
\begin{equation} \label{curv2}
\langle R^{\perp}(X,Y)\xi_1,\xi_2 \rangle=\langle [A_{\xi_1},A_{\xi_2}]X,Y \rangle
\end{equation}
for $X,Y$ in $TM$ and $\xi_1, \xi_2$ in $NM$, where $\langle A_{\xi}X,Y \rangle=-\langle \alpha_XY, \xi \rangle$ is the shape operator. It shows that also $(\im \alpha )^{\perp}=\{\xi \in NM: A_{\xi}=0\}$ is annihilated by $R^{\perp}(E_0,TM)$ and (ii) follows.
\end{proof}

\begin{pro} \label{ll2}
Let $(M,g)$ be a $k$-parallel isometric immersion into a space form:
\begin{itemize}
 \item[(i)] we have that $(M,g)$ is semi-parallel, that is, $R(X,Y)\cdot \alpha=0$ whenever $X,Y$ belong to $TM$;
\item[(ii)] and $ \nabla R^{\perp}=0.$
\end{itemize}
\end{pro}
\begin{proof}
(i) From the lemma and the definition of $E_0$ it follows that the tensor 
\begin{equation*}
(R(X,Y) \cdot \alpha)(Z_1,Z_2)=R^{\perp}(X,Y)(\alpha_{Z_1}Z_2)-\alpha(R(X,Y)Z_1,Z_2)-\alpha(Z_1,R(X,Y)Z_2)
\end{equation*}
vanishes when $X$ is in $E_0$. For $X,Y$  in $E_1$ we have $\nabla^2_{X,Y} \alpha=0$ by (iii) in Proposition \ref{s1} and 
$R(X,Y) \cdot \alpha=0$ follows.\\
(ii) By differentiation in the lemma above 
$$ (\nabla_{TM}R^{\perp})(E_0,TM)=0$$
since $E_0$ is parallel inside $TM$. There remains to show that $(\nabla_{TM}R^{\perp})(E_1,E_1)=0$.
For $X,Y$ in $E_1, U$ in $TM$ as well as $\xi_1, \xi_2$ in $NM$ we have 
\begin{equation*}
\langle (\nabla_{U}R^{\perp})(X,Y)\xi_1, \xi_2 \rangle=\langle [(\nabla_UA)_{\xi_1},A_{\xi_2}]X,Y\rangle+
\langle [A_{\xi_1},(\nabla_UA)_{\xi_2}]X,Y\rangle
\end{equation*}
by differentiation in the Ricci equation \eqref{curv2}. Moreover
\begin{eqnarray*}
\langle [(\nabla_UA)_{\xi_1},A_{\xi_2}]X,Y\rangle&=&\langle (\nabla_UA)_{\xi_1}A_{\xi_2}X,Y \rangle-
\langle A_{\xi_2}(\nabla_UA)_{\xi_1}X,Y \rangle\\
&=&\langle (\nabla_UA)_{\xi_1}Y,A_{\xi_2}X\rangle-\langle A_{\xi_2}(\nabla_UA)_{\xi_1}X,Y \rangle.
\end{eqnarray*}
Now for space forms the right hand side in the Codazzi-Mainardi equation \eqref{c-m} vanishes; 
since $\nabla_{E_1}A=0$ by (iii) in Proposition \ref{s1} each summand above is zero. A similar argument shows that the same holds for the second commutator in the expression of $\nabla R^{\perp}$ and the claim is proved.
\end{proof}
We consider 
$$ N_0M=\{\xi \in NM:R^{\perp}(TM,TM)\xi=0\}
$$
with orthogonal complement $N_1M$ in $NM$. 
\begin{pro} \label{pro3}
Let $(M,g)$ be a $k$-parallel isometric immersion into a space form. We have:
\begin{itemize}
 \item[(i)] $N_0M, N_1M$ are parallel;
\item[(ii)] $N_0M$ is flat.
\end{itemize}
\end{pro}
\begin{proof}
(i) is a direct consequence of (ii) in Proposition \ref{ll2} while (ii) follows from (i) and the definition of $N_0M$.
\end{proof}

Let us consider the following subbundle of the normal bundle of $M$
$$
F_0:= \span \{(\nabla^l_{X_1,\dots, X_l} \alpha)(Y,Z) : 1\le l\le k-1, X_1,\dots,X_l,Y,Z \in TM \} \subseteq NM.
$$
By applying Proposition \ref{s1}, (iii) and Lemma \ref{ll1}, (i) and using the Codazzi-Mainardi equation we can also write this as 
\begin{equation}\label{F02} 
F_0= \span \{(\nabla^l_{X_1,\dots, X_l} \alpha)(Y,Z) : 1\le l\le k-1, X_1,\dots,X_l,Y,Z \in E_0 \} \subseteq NM.
\end{equation}
\begin{pro} \label{x} Let $(M,g)$ be a $k$-parallel isometric immersion into a space form. The following hold: 
\begin{itemize}
\item[(i)] the subbundle 
$ F_0 \subseteq NM$
is parallel w.r.t. the normal connection;
\item[(ii)] the bundle $F_0$ is flat, that is $R^{\perp}(TM,TM)F_0=0$.
\end{itemize}
\end{pro}
\begin{proof}
(i) is a direct consequence of $\nabla^k\alpha=0$.\\ 
(ii) We use (\ref{F02}) for the description of $F_0$. Let $X_1,\dots,X_l,Y,Z$ be in $E_0$ and $1\le l\le k-1$.  We have  $R(U,V)\cdot \nabla^l\alpha=0$ for all $U, V\in TM$ by differentiating the equation $R(U,V)\cdot \alpha=0$, which we know from Proposition \ref{ll2}, (i). This gives
\begin{eqnarray*}
\lefteqn{R^{\perp}(U,V)((\nabla^l_{X_1,\dots, X_l} \alpha)(Y,Z))=\sum_{j=1}^l (\nabla^l_{X_1,\dots,R(U,V)X_j,\dots, X_l}\alpha)(Y,Z)}\\&& +(\nabla^l_{X_1,\dots, X_l} \alpha)(R(U,V)Y,Z)+(\nabla^l_{X_1,\dots, X_l} \alpha)(Y, R(U,V)Z)
\end{eqnarray*}
Since $E_0$ is flat, the right hand side of the above equation vanishes. \end{proof}
\begin{lema} \label{alpha} If $(M,g)$ is a $k$-parallel isometric immersion into a space form, then
 \begin{itemize}
  \item[(i)] $\alpha(E_0,E_0) \subseteq N_0M$; 
\item[(ii)] $\alpha(E_0,E_1) \subseteq N_1M$;
\item[(iii)] $A_{N_0M}E_i\subseteq E_i$, $i=1,2$;
\item[(iv)]  $A_{F_0}E_1=0$.
 \end{itemize}
\end{lema}
\begin{proof} (i) Since $(M,g)$ is semiparallel by Proposition \ref{ll2} and since $E_0$ is flat we have
$$R(U,V)^\perp( \alpha(X,Y))= \alpha(R(U,V)X,Y)+\alpha(X, R(U,V)Y)=0$$
for all $X,Y\in E_0$ and all $U,V\in TM$.

(ii)  Using again that $(M,g)$ is semiparallel and that $E_0$ is flat we obtain
$$R^\perp(U,V)( \alpha(X,Y))= \alpha(R(U,V)X,Y)$$
for all $U,V,X\in TM$ and $Y\in E_0$. The claim now follows from $E_1=\span \{ R(X,Y)Z : X,Y,Z \in TM\}$.

(iii) follows from (ii). Indeed,
$$\langle A_{N_0M} E_1, E_0\rangle = \langle E_1, A_{N_0M} E_0\rangle = \langle \alpha(E_1,E_0),N_0M\rangle =0.$$

(iv) We have to show $\alpha(E_1,TM)\perp F_0$. Obviously, $\alpha(E_1,E_0)\perp F_0$ by (ii) since Proposition \ref{x}, (ii) gives $F_0\subset N_0M$. It remains to show $\alpha(E_1,E_1)\perp F_0$. Take $Y_0,Z_0\in E_0$ and $Y_1,Z_1\in E_1$.  
Because $R(E_0,E_1,E_0,E_1)=0$ the Gauss equation \eqref{g1} yields 
\begin{equation*}
\langle \alpha(Y_1,Y_0), \alpha(Z_1,Z_0) \rangle-\langle \alpha(Y_1,Z_1), \alpha(Y_0,Z_0) \rangle=c \langle Y_0,Z_0\rangle 
\langle Y_1,Z_1\rangle .
\end{equation*}
After differentiation, using that $\nabla_{E_1}\alpha=0$ the Codazzi-Mainardi equation leads to 
$$\langle \alpha(Y_1,Z_1), (\nabla^l_{X_1,\dots, X_l} \alpha)(Y_0,Z_0)\rangle=0$$
for all $1\le l\le k-1$ and $X_1,\dots, X_l\in E_0$. Now  it follows from (\ref{F02}) that $\alpha(E_1,E_1)$ is orthogonal to $F_0$. 
\end{proof}
Now we will define a further decomposition of $E_0$ into subbundles $E_0'\oplus E_0''$ such that Moore's criterion applies to $D_0:=E_1\oplus E_0'$ and $D_1:=E_0''$.
Let $E_0^{\prime \prime}$ be spanned by elements of the form 
$$ A_{\xi}X,\ (\nabla^l_{U_1,\dots U_l}A)_{\xi}X,\quad l=1,\dots,k-1, 
$$
where $\xi$ belongs to $F_0$, $X$ is in $TM$ and $U_1,\dots,U_l$ are in $TM$. Equivalently, it suffices to take $X\in E_0$. Indeed, this follows from Lemma \ref{alpha}, (iv).

Obviously, $E_0''\subset E_0$ by Lemma \ref{alpha}, (iii). We will denote by $E_0^{\prime}$ the orthogonal complement of $E_0^{\prime \prime}$ in $E_0$.  

\begin{pro} \label{las}
If $(M,g)$ is a $k$-parallel isometric immersion into a space form, then
\begin{itemize}
\item[(i)] $E_0'$ and $E_0^{\prime \prime}$ are parallel w.r.t. $\nabla$;
\item[(ii)] $E_0^{\prime} \subseteq \{ X \in E_0: A_{F_0}X=0 \}$;
\item[(iii)]$E_0^{\prime}\subseteq \{X \in E_0: \nabla_X \alpha=0 \}$;
\end{itemize}
\end{pro}
\begin{proof}
(i) follows directly from having $E_0$ and $F_0$ parallel and $\nabla^k \alpha=0$. 

(ii) is a direct consequence of the definition of $E_0'$.

(iii) Take $X\in E_0'$. We first note that $(\nabla_UA)_\xi X=0$ hence also $\langle (\nabla_U\alpha)(X,\cdot),\xi\rangle=0$ holds for all $U\in TM$ and $\xi \in F_0$ by the definition of $E_0'$. By the Codazzi-Mainardi equation we obtain $\langle (\nabla_X\alpha)(U,\cdot),\xi\rangle=0$ for all $U\in TM$ and $\xi \in F_0$. Since $(\nabla_X\alpha)(TM,TM)\subseteq F_0$ the claim follows.
\end{proof}
\begin{pro} \label{last}
$\alpha(E_1 \oplus E_0^{\prime}, E_0^{\prime \prime})=0$. 
\end{pro}

\begin{proof} 
The orthogonal complement $F_1$ of $F_0$ clearly contains $N_1M$. Take $\xi_1\in F_1$. Then 
$$
\langle (\nabla^l_{U_1,\dots, U_l} A)_{\xi_1} X,Y\rangle = \langle (\nabla^l_{U_1,\dots, U_l}\alpha)(X,Y),\xi_1\rangle=0
$$
holds for all $X,Y,U_1,\dots, U_l\in TM$ since $(\nabla^l_{U_1,\dots, U_l}\alpha)(X,Y)\in F_0 $. Thus
\begin{equation}\label{xx}(\nabla^l_{U_1,\dots, U_l} A)_{\xi_1}=0
\end{equation}
The Ricci equation (\ref{curv2}) gives 
\begin{equation}\label{x1}[A_\xi,A_{\xi_1}]=0\end{equation}
 for all $\xi\in N_0M$ and $\xi_1\in F_1$, which implies 
\begin{equation}\label{x3}
 A_{\xi_1}(\im A_\xi)\subseteq \im A_\xi.
\end{equation}
 Differentiating (\ref{x1}) and taking into account (\ref{xx}) we obtain 
$$[(\nabla^l_{U_1,\dots, U_l} A)_{\xi},A_{\xi_1}]=0.$$
This gives
\begin{equation}\label{x2}A_{\xi_1} (\im (\nabla^l_{U_1,\dots, U_l} A)_{\xi}) \subseteq \im (\nabla^l_{U_1,\dots, U_l} A)_{\xi}\end{equation}
for all $U_1,\dots, U_l\in TM$.  Applying (\ref{x3}) and (\ref{x2}) to all $\xi\in F_0$ we obtain
\begin{equation} \label{x4}
A_{F_1} (E_0'')\subseteq E_0''.
\end{equation}
Since $N_1M \subseteq F_1$ this implies 
$$ \langle \alpha(E_0^{\prime \prime}, E_1),N_1M\rangle=\langle A_{N_1M}E_0^{\prime \prime}, E_1\rangle=0 
$$
which combined with $\alpha(E_0,E_1) \subseteq N_1M$ shows that $\alpha(E_0^{\prime \prime},E_1)=0$. 

At the same time from \eqref{x4}
$$ \langle \alpha(E_0^{\prime}, E_0^{\prime \prime}),F_1\rangle=\langle A_{F_1}E_0^{\prime \prime}, E_0^{\prime}\rangle=0 
$$
which shows that $ \alpha(E_0^{\prime}, E_0^{\prime \prime})$ is contained in $F_0$; however, $\langle \alpha(E_0^{\prime},
E_0''),F_0\rangle=\langle A_{F_0}E_0', E_0''\rangle=0 $ by (ii) in Proposition \ref{las} leads to its vanishing.

\end{proof}
$\\$
{\bf{Proof of Theorem \ref{main1}}}\\
The orthogonal splitting $TM=(E_1 \oplus E_0') \oplus E_0^{\prime \prime}$ is parallel w.r.t. to the Levi-Civita 
connection and hence induces a local splitting of $M$ with corresponding factors $M_1$ and $M_2$. Proposition 
\ref{last} makes that Molzan's generalisation of Moore's Theorem applies to 
the immersion of $M_1 \times M_2$ in $M^n(c)$; it yields a product decomposition of $f$ according to 
Definition \ref{def-0} with factors $f_i:M_i \to N_i, i=1,2$. 

From (iii) in Propositions \ref{s1} and \ref{las} it follows that $(\varphi_1 \circ f_1)(M_1)$ is parallel 
in $M^n(c)$. Since the image of $N_1$ in $M^n(c)$ is spherical it follows by a standard argument that 
$(\varphi_1 \circ f_1)(M_1)$ is parallel in $\varphi_1(N_1)$ as well. Hence $f_1:M_1 \to N_1$ is parallel since 
$\varphi_1:N_1 \to \varphi_1(N_1)$ is an isometry.

Since the distributions $E_0$ and hence $E_0''$ are flat so is $M_2$ and the same argument as above shows that 
it is immersed in $N_2$ as a $k$-parallel immersion.

In absence of a factor of type $E_1$ it follows 
from (ii) in Lemma \ref{ll1} that the normal bundle of $M_2$ is flat. $\square$

\begin{acknowledgements}
It is a pleasure to thank A.\,J.\,di Scala for useful comments on an earlier version of this paper.
\end{acknowledgements}

\end{document}